\documentclass{amsart}
\usepackage[utf8]{inputenc}

\usepackage{amssymb}
\usepackage{epsfig}
\usepackage{amsfonts}
\usepackage{amsmath}
\usepackage{euscript}
\usepackage{amscd}
\usepackage{amsthm}
\usepackage{enumerate}

\DeclareMathAlphabet{\mathpzc}{OT1}{pzc}{m}{it}
\usepackage{mathrsfs}

\DeclareSymbolFont{SY}{U}{psy}{m}{n}

\usepackage{tikz-cd}

\usepackage{ wasysym }

\usepackage{stmaryrd}

\theoremstyle{plain}

\newtheorem{thm}{Theorem}[section]
\newtheorem*{thm*}{Theorem}
\newtheorem{cor}[thm]{Corollary}
\newtheorem{lem}[thm]{Lemma}
\newtheorem{prop}[thm]{Proposition}
\theoremstyle{definition}
\newtheorem{defn}[thm]{Definition}
\newtheorem{rem}[thm]{Remark}
\newtheorem{ex}[thm]{Example}

\numberwithin{equation}{section}
\def\A{{\mathcal A}}
\def\C{{\mathbb C}}
\def\D{{\mathbb D}}

\def\cH{{\mathcal{H}}}

\def\K{{\mathscr K}}
\def\B{{\mathcal{B}}}

\def\norm#1{\left\lVert{#1}\right\rVert}
\def\setof#1{\{{#1}\}}

\def\inp#1,#2{\left\langle{#1},{#2}\right\rangle}

\def\bN{\mathbb{N}}

\def\I{\mathcal{I}}
\def\J{\mathcal J}

\def\m{\mathcal}

\def\P{\mathcal{P}}

\def\beq{\begin{eqnarray}}
\def\eeq{\end{eqnarray}}
\def\beqa{\begin{eqnarray*}}
	\def\eeqa{\end{eqnarray*}}

\def\T{\mathbb{T}}

\def\<{\langle}
\def\>{\rangle}
\def\Z{\mathbb{Z}}

\def\*{*}

\def\P{\mathcal{P}}

\def\m*{m^\*}
\def\cO{\mathcal{O}}
\def\bF{\mathbb{F}}

\def\P{\mathscr{P}}

\def\cC{\mathcal{C}}
\def\ci{{\iota}}

\def\abs#1{\left\lvert{#1}\right\rvert}

\def\st{\textrm{ } \; \vert \; \textrm{ }}

\def\cI{\mathcal{I}}

\newcounter{cnt1}
\newcounter{cnt2}
\newcounter{cnt3}
\newcommand{\blr}{\begin{list}{$($\roman{cnt1}$)$}
		{\usecounter{cnt1} \setlength{\topsep}{0pt}
			\setlength{\itemsep}{0pt}}}
	\newcommand{\bla}{\begin{list}{$($\alph{cnt2}$)$}
			{\usecounter{cnt2} \setlength{\topsep}{0pt}
				\setlength{\itemsep}{0pt}}}
		\newcommand{\bln}{\begin{list}{$($\arabic{cnt3}$)$}
				{\usecounter{cnt3} \setlength{\topsep}{0pt}
					\setlength{\itemsep}{0pt}}}
			\newcommand{\el}{\end{list}}

\usepackage{geometry}
\usepackage{hyperref}
\hypersetup{
    colorlinks=true,
    linkcolor=blue,
    filecolor=blue, 
    citecolor=blue,
    urlcolor=red
    }

\allowdisplaybreaks[1]

\usepackage{orcidlink}

\title[Dilating Semigroup Representations to the Boundary Quotient]{Dilating Semigroup Representations to the Boundary Quotient}

\author[Ujan Chakraborty]{Ujan Chakraborty \orcidlink{0000-0002-2655-4488}
}
\address{School of Mathematics and Statistics, University of Glasgow, University Avenue, Glasgow G12 8QQ, UK}
\email[Ujan Chakraborty]{Ujan.Chakraborty@glasgow.ac.uk}



\date{}

\begin{document}

\begin{abstract}
    We prove that a representation of a group embeddable or right LCM cancellative semigroup may be dilated to a representation of its reduced boundary quotient $C^\star$-algebra if and only if it extends to a completely contractive representation of the reduced semigroup operator algebra. 
    We show that the latter property is satisfied not only by all constructible representations of amenable semigroups, but also a very large class of other representations, which encompasses several classical dilation theorems and the corresponding matricial von Neumann inequalities. 
    In fact, our criterion of completely contractive extension to the operator algebra turns out to be an appropriate generalisation of the matricial von Neumann inequality to semigroups more general than $\bN^k$, and dilation to the boundary quotient turns out to be an apt generalisation of unitary dilations. 
    Thus, our theorem is in spirit and in practice a generalisation of Sz.-Nagy and Ando's dilation theorems for general semigroups. 
    In addition, this also demonstrates that any completely contractive representation of the operator algebra dilates to a representation with additional relations among its generators, the new relations coming from the boundary quotient. 
    In particular, for Ore semigroups, this completely characterises which representations admit unitary dilations.  
\end{abstract}

\maketitle

\section{Introduction}

We show in this paper that if $P$ is a group-embeddable or right LCM cancellative semigroup, then every representation of $P$ that extends to a completely contractive representation of the non-selfadjoint operator algebra generated by the left-regular representation $\A_\lambda(P)$, dilates to a representation of the reduced boundary quotient $C^\star$-algebra $\partial C^\star_\lambda(P)$. 
The beauty of our technique is that it is quite a general dilation theoretic idea that we use. 
While the first mentioned property (complete contractivity) might be difficult to verify in general, it turns out to be an appropriate generalisation of the matricial von Neumann's inequality for $\bN^k$. 
And while the boundary quotient is sometimes difficult to compute, for Ore semigroups, it is just the $C^\star$-algebra of the enveloping group, and thus dilation to the boundary quotient turns out to be an appropriate generalisation of the notion of unitary dilation. 
This allows us to bring existing well-known dilation theorems like those of B. Sz.-Nagy \cite{SN53}, T. Ando \cite{An63}, and J. A. Holbrook \cite{Ho92} under the same umbrella, provides necessary and sufficient conditions for unitary dilations of Ore semigroup representations, and demonstrates that they are all manifestations of the same dilation phenomenon. 
This is considerably more general than existing generalisations of Sz.-Nagy's dilation theorem, like the results by S. Brehmer \cite{Br61}.

The general goal of dilation theory is to view maps between spaces as parts of nicer maps between potentially larger spaces. 
For example, contractions on Hilbert spaces may be dilated to unitaries on larger Hilbert spaces, upto powers \cite{SN53}: for any contraction $T$ on a Hilbert space $\cH$, there exists a unitary $U$ on a bigger Hilbert space $\K \supseteq \cH$ such that $T^n = \P_\cH U^n \vert_\cH$ 
for all $n \in \bN$
. 
This can be done for commuting pairs of contractions as well \cite{An63}. 
However, such a result fails to hold more generally for arbitrary commuting triples of contractions, as noted in \cite{Pa70}, unless they are $2 \times 2$ complex matrices, in which case power dilations work for all $k$-tuples \cite{Ho92}. 
This leads one to the following question, the answer to which is 
well-known:
\newline\textbf{Question A: } Under what conditions may commuting $k$-tuples of contractions be dilated to commuting $k$-tuples of unitaries upto powers? 
\newline\textbf{Answer A: } A contractive $k$-tuple of commuting contractions has a commuting unitary dilation upto powers if and only if it satisfies a matricial von Neumann inequality. 
\newline
Since $k$ commuting contractions induce a contractive representation of $\bN^k$, we realise that perhaps this question is answered in greater generality by viewing the representations of semigroups (in this case $\bN^k$, but potentially a lot more general) as (non-selfadjoint) operator algebras, and their dilation theory. 
For example, one may also formulate the following question:
\newline\textbf{Question B: } Under what conditions may a contractive representation of a semigroup be dilated to a unitary representation? 
Or, more generally, what might be the counterpart of a unitary dilation?

As it turns out, the boundary quotient is the correct choice for generalising unitary dilations. 
We briefly recall some of the motivation behind dilation to the ``boundary''. 
Consider the disk algebra $\A(\D)$, holomorphic functions on the unit disk $\D$ which are bounded and hence extend to continuous functions on its closure $\bar{\D}$. 
They form a non-selfadjoint algebra. 
Their closure under involution $C(\bar{\D})$ is a $C^\star$-algebra, but it is not the ``smallest'' $C^\star$-algebra into which $\A(\D)$ embeds isometrically. 
The right candidate for the latter would be all continuous functions on the unit circle, $C(\T)$, and one may verify this by observing that by the maximum modulus principle, the map that sends a function in $\A(\D)$ to its restriction on $\T$ is in fact an isometry. 
One notes that $\T$ is the (\v{S}ilov) boundary
, and this phenomenon generalises classically to all homogeneous algebras. $C^\star$-algebras.
This was one of the motivations behinds the development of the theory of $C^\star$-envelopes. 
For a (possibly noncommutative and non-selfadjoint) operator algebra $\A$, one may consider $C^\star_{env}(\A)$ to be the ``smallest'' $C^\star$-algebra into which $\A$ embeds completely isometrically. 
$C^\star$-envelopes were first defined by W. Arveson \cite{Ar69}, and he conjectured their existence in general.
In case of unital operator algebras, the existence of the $C^\star$-envelope was proved by M. Hamana through injective envelopes \cite{Ha79a} and M. Dritschel and S. A. McCullough through dilation-theoretic techniques \cite{DM05}. 
The non-unital case follows from R. Meyer's work \cite{Me01}. 
The Dritschel-McCullough dilation theorem shows that every completely contractive representation of $\A$ has a maximal dilation, and the $C^\star$-algebra generated by a maximal dilation of a completely isometric representation of $\A$ is $C^\star_{env}(\A)$, which is in general ``smaller'' than $C^\star(\A)$ (in fact, a quotient), in a similar fashion as $C(\T)$, the $C^\star$-envelope of $\A(\D)$, is a quotient of $C(\bar{\D})$. 
Moreover, ever completely contractive representation of $\A$ dilates to a representation of $C^\star_{env}(\A)$. 

One makes a connection here, that $\A(\bar \D)$ is in fact the operator algebra generated by a representation of the semigroup $\bN$ on the Hilbert space $L_2(\bar \D)$ that maps the generator $1$ to $z$. 
So, one might venture to ask what a suitable notion of boundary might be for representations of semigroups in general. 
We point out that in the context of representations of semigroups, one has a quotient object, the boundary quotient, which is in general smaller than the $C^\star$-algebra generated by the left-regular representation of the semigroup. 
The representation theory of semigroups is considerably more involved than their counterpart for groups, in the same fashion as non-selfadjoint operator algebras often offer more flexibility (and hence, hide greater complexity) in their representations than their selfadjoint counterparts. 
In constructing both universal and reduced $C^\star$-algebras of semigroups ($C^\star_s(P)$ and $C^\star_\lambda(P)$ respectively), one must account for their ideal structure \cite{Li12a}. 
Even then, one looks for ``smaller'' $C^\star$-algebras which encode complete information about $P$, and this leads us to the boundary quotient. 
For $\bN^k$, the (full and reduced) boundary quotient is isomorphic to $C(\T^k) \cong C^\star(\Z^k)$, which is a quotient of $C^\star(\bN^k)$, a much bigger $C^\star$-algebra (the Toeplitz algebra in the $k=1$ case). 
Similar phenomena are seen for free semigroups $\bF_k^+$, where the $C^\star$-algebra of the semigroup is the Cuntz-Toeplitz algebra $\mathcal{TO}_k$, while the boundary quotient is the Cuntz algebra $\mathcal{O}_k$, a quotient of $\mathcal{TO}_k$. 
This phenomenon was generalised to right-angled Artin monoids by J. Crisp and M. Laca in \cite{CL07}. 
Their construction was further extended to more general semigroups in two directions, one through full and reduced partial crossed products (e.g. \cite{LS22, KKLL22}), the other through quotienting these $C^\star$-algebras by products of complements of projections of ideals which intersect every ideal non-trivially (called the boundary quotient relations) \cite{BRRW14}. 
If $P$ embeds into a group $G$, one has the co-universal property that reduced boundary quotient $\partial C^\star_\lambda(P)$ is the ``smallest'' $C^\star$-algebra generated by a $G$-equivariant representation of $C^\star_s(P)$ \cite{KKLL22}. 
C. Sehnem proved in \cite{Se22} that $C^\star_{env}(\A_\lambda(P))$ is canonically isomorphic to $\partial C^\star_\lambda(P)$ if $P$ is group-embeddable, and the same result was proved without the group-embeddability assumption by K. Brix, C. Bruce, and A. Dor-On in \cite{BBD26} in the right LCM setting, assuming cancellation. 

And this finally leads us to our dilation result. 
If a representation of a semigroup $P$ extends to a completely contractive representation of $\A_\lambda(P)$, then we may use the Dritschel-McCullough dilation machine to dilate it to a representation of $C^\star_{env}(P)$. 
Then, if the criterion of \cite{Se22, BBD26}, are met, we may dilate it to a representation of $\partial C^\star_\lambda(P)$, and thus we have our main theorem: 
\begin{thm*}[Theorem \ref{thm: representations with property (A) extend to the boundary}]
    Let $P$ be a group-embeddable or right LCM cancellative semigroup. 
    Then, a representation of $P$ dilates to a representation of $\partial C^\star_\lambda(P)$ if and only if it extends to a completely contractive representation of $\A_\lambda(P)$.
\end{thm*}

This immediately allows us to answer question $B$ posed above. 
By Theorem \ref{thm: boundary quotients for right-Ore semigroups}, we have that for a semigroup $P$ with enveloping group $G$, $\partial C^\star_\lambda(P)$ is isomorphic to $C^\star_\lambda(G)$ if and only if $P$ is right Ore. 
Given that group representations are unitary, this allows us to answer Question B:
\newline\textbf{Answer B: } If $P$ is Ore, then every contractive representation that extends to a completely contractive representation of $\A_\lambda(P)$, has a unitary dilation. 
\newline
Sufficient conditions for representations to dilate to $\mathcal{Q}(P)$ in the right LCM setting were obtained by M. Laca and B. Li in \cite{LL22} through constructive methods, but their conditions are very strict, and are not satisfied by many representations, including the classical ones in \cite{An63, Ho92}. 
But the conditions we require for the Dritschel-McCoullough dilation machinery to operate for the dilation to $\partial C^\star_\lambda(P)$ are more general than theirs, and not just sufficient but also necessary. 
This allows one to reconcile with some existing operator-theoretic dilation results like the ones mentioned above. 
The only drawback is that it might be difficult to check complete contractivity in specific cases, and may still need heavy operator-theoretic machinery. 
But we show that this condition is equivalent to the matricial von Neumann inequality in the case of $\bN^k$, and hence is not expected to require techniques considerably more difficult than existing ones, and is in fact more generally applicable. 

This relation between matricial norm properties of a representation of a semigroup and the existence of a dilation to its boundary quotient might seem unexpected, and makes sense only through reflection into the nature of the Dritschel-McCullough dilation machine, and the equivalence of the $C^\star$-envelope and the boundary quotient. 
Meanwhile, one may note that though the semigroup and its boundary quotient are generated by the same elements, the latter might have many more relations, simply due to the fact that the former does not have inverses in general while the latter has adjoints which must respect the ideal structure of the former. 
These are precisely the foundation set relations. 
This leads us to the Corollary \ref{cor: property A guarantees that there exists a dilation with boundary quotient relations}, that every representation that extends to a completely contractive representation of $\A_\lambda(P)$, also admits a dilation with these additional relations among the generators. 

The structure of the paper is the following:
in section 2, we provide dilation-theoretic preliminaries, and introduce the notions of envelopes and boundaries. 
In section 3, we provide background on semigroups, while 4 and 5 are background on semigroup $C^\star$-algebras and boundary quotients respectively. 
Finally, in section 6, we demonstrate the dilation phenomenon.

\section*{Acknowledgements}

We wish to thank Chris Bruce, Adam Dor-On, and Xin Li 
for various insightful comments.  
We also wish to thank Joachim Zacharias, Boyu Li, Mike Whittaker, Charles Starling, and Runlian Xia for helpful comments on the draft.

\section{Preliminaries on Dilation Theory, Envelopes, and Boundaries}

The general goal of dilation theory is to view an operator (or a map into a space of operators) as a part of something simpler on a larger space. 
Its origins can be traced to operator dilation theory stemming from B. Sz{\H{o}}kefalvi-Nagy's celebrated dilation theorem \cite{SN53}. 
In this paper, we shall be dealing principally with dilations of representations. 
For us, an \textit{operator algebra} $\A$ will be a norm-closed subalgebra of $\B(\cH)$, not necessarily self-adjoint. 
Furthermore, we shall assume our operator algebras to be unital, and our representations to be unital as well. 
We refer the reader to K. R. Davidson's recent textbook \cite{Da25} for a detailed account on some of the topics we briefly mention in this section, and O. M. Shalit's notes \cite{Sh21} for a detailed introduction to dilation theory. 
For a general introduction to completely bounded maps we refer to V. I. Paulsen's celebrated textbook \cite{Pa02}. 

\begin{defn}\label{defn: C*-cover, C*-envelope}
    Given an operator algebra $\A$, we call $(\cC, \ci)$ a \textit{$C^\star$-cover} of $\A$, if $\ci : \A \to \cC $ is a completely isometric representation with $\cC = C^\star (\ci(\A))$.
    The \textit{$C^\star$-envelope} of $\A$ is a $C^\star$-cover $(C^\star_{env}(\A), \ci)$ that has the following co-universal property: if $(\cC', \ci')$ is another $C^\star$-cover of $\A$, then there exists a (necessarily unique) $\star$-epimorphism $\Phi: \cC' \to C^\star_{env}(\A)$ such that $\Phi(\ci'(a)) = \ci(a)$ for all $a \in \A$. 
\end{defn}

We should mention that there is a high amount of non-uniformity in terminology regarding the terms introduced in the rest of this section. 
What we refer to as ``maximal'' dilations, are occasionally called ``extremal'', or $\partial$-representations. 
We refer the reader to \cite{DM05, Ar03, Ar08, Da25} for further reading. 

\begin{defn}\label{defn: dilation, maximal dilation}
    Let $\A$ be an operator algebra. 
    Given a representation  $\phi: \A \to \B(\cH)$, another representation $\phi ': \A \to \B(\cH ')$ is called a \textit{dilation} of $\phi$ if $\cH \subseteq \cH '$, and $\phi(a) = \P_\cH \phi '(a) \vert_\cH$ for all $ a \in \A$. 
    A completely contractive representation $\phi$ is called \textit{maximal} if every dilation is trivial, that is, for every dilation $\phi': \A \to \B(\cH')$, $\P_\cH \phi'(a) = \phi'(a) \vert_\cH = \phi(a)$ for all $a \in \A$ (which is equivalent to saying that one may decompose $\cH' = \cH \oplus \cH''$ such that $\phi' = \phi \oplus \phi''$ for some representation $\phi''$). 
\end{defn}

\begin{thm}[Theorem 1.1, \cite{DM05}]\label{thm: Unique Extension Property}
    Let $\A$ be an operator algebra. 
    Then, the following are equivalent:
    \begin{enumerate}
        \item[(i)] $\phi: \A \to \B(\cH)$ is a maximal representation. 
        \item[(ii)] For every $C^\star$-cover $(\mathcal{C}, \iota)$ of $\A$, there exists a representation $\psi: \mathcal{C} \to \B(\cH)$ satisfying $\psi \circ \iota = \phi$, and the only completely positive map agreeing with $\psi$ when restricted to $\iota(\A)$ is $\psi$ itself. 
    \end{enumerate}
\end{thm}

The second condition in Theorem \ref{thm: Unique Extension Property} is called the unique extension property. 

\begin{thm}[Parts of theorems 1.2 and 4.1, \cite{DM05}]\label{thm: Dritschel-McCoulough theorem}
    Let $\A$ be an operator algebra, and $\phi: \A \to \B(\cH)$ a completely contractive representation. 
    Then, $\phi$ has a maximal dilation. 
    Moreover, if $\phi$ is completely isometric, and $\phi'$ is its maximal dilation, then $C^\star(\phi'(\A)) \cong C^\star_{env}(\A)$. 
\end{thm}

This theorem also proved that the $C^\star$-envelope exists for every operator algebra. 
We are yet to properly define what a boundary means in this context. The basic examples of $C^\star$-envelopes, as well as motivation for boundaries, arise in the context of uniform algebras. 

\begin{ex}\label{ex: C*-envelope of the disk algebra}
    Recall that we observed in the introduction, for the disk algebra $\A(\D)$, we have $C^\star_{env}(\A(\D)) \cong C(\T)$, the quotient map mapping every bounded holomorphic function to a function on the unit circle obtained by restricting its domain to the latter, and the maximum modulus ensuring that this is an isometry. 
\end{ex}

The result in Example \ref{ex: C*-envelope of the disk algebra} is not an isolated one. 
The $C^\star$-envelope of any uniform algebra is the $C^\star$-algebra of continuous functions on its {\v S}ilov boundary. 
The unconditional existence of the $C^\star$-envelope (Theorem \ref{thm: Dritschel-McCoulough theorem}) provides a noncommutative analogue of this result. 

\begin{defn}\label{defn: Shilov ideal}
    Let $\A$ be an operator algebra. 
    An ideal $\cI \lhd C^\star(\A)$ is called a \textit{boundary ideal} if the quotient map $q_\cI:C^\star(\A) \to C^\star(\A) / \cI$ restricts to a complete isometry on $\A \subseteq C^\star(\A)$. 
    The \textit{{\v S}ilov ideal} $\cI_s$ of $\A$ is defined to to be the boundary ideal that contains all boundary ideals of $\A$. 
\end{defn}

The existence of the {\v S}ilov ideal can be seen from the existence of $C^\star$-envelope: $C^\star_{env}(\A)$ is canonically isomorphic to $C^\star(\A)/\cI_s$.

\section{Preliminaries on Semigroups}

A semigroup is a set with an associative binary operation, traditionally referred to as a multiplication. 
In this paper that by semigroup we mean a semigroup with an identity element, that is, a monoid. 
For example, we shall in this spirit assume $(\mathbb{N},+)$ to mean the additive natural numbers with $0$. 
All semigroups under consideration in this paper will be discrete and countable. 
We begin with the following essential property: 
\begin{defn}\label{defn:cancellative}
    A semigroup $P$ is said to be \textit{left cancellative} if for any $p,q,x \in P$, $xp = xq \implies p = q$. 
    Correspondingly, we say $P$ is \textit{right cancellative} if for any $p,q,x \in P$, $px = qx \implies p = q$. 
    We say $P$ is \textit{cancellative} if it is left and right cancellative. 
\end{defn}

We denote the set of invertible elements (also called units) of $P$ by $P^\star$. 
Every group is (trivially) a semigroup, and it is interesting to consider which conditions on a semigroup would ensure that it is (faithfully) embeddable in a group. 
A necessary but not sufficient condition for this is the cancellative property. 
A. Mal'cev gave an infinite list of conditions which are necessary and sufficient for group-embeddability of a semigroup, and showed that no finite list of those conditions suffices. 
We refer the reader to Section 12 of \cite{CP67} for further details. 
We mention one particular situation in which $P$ is group-embeddable, which will be useful later. 
We denote $pP := \setof{px \; \vert \; x \in P}$ and define $Pp$ in a similar fashion. 
\begin{defn}\label{defn: Ore}
    We say a semigroup $P$ is \textit{right-reversible} if for all $p, q \in P$, $Pp \cap Pq \neq \emptyset$; and 
    \textit{left-reversible} if for all $p, q \in P$, $pP \cap qP \neq \emptyset$. 
%
    $P$ is called \textit{left Ore} if it is cancellative and right-reversible; and 
    \textit{right Ore} if it is cancellative and left-reversible. 
\end{defn}

\begin{thm}\label{thm: Ore embedding}[Ore, Dubreil]
    A semigroup $P$ embeds into a group $G$ of the form $G = P^{-1} P$ if and only if it is left-Ore. 
    Similarly, $P$ embeds into a group $G$ of the form $G = P P^{-1}$ if and only if it is right-Ore. 
\end{thm}

We direct the reader to Theorem 1.24 in \cite{CP61} or Section 1.1 of \cite{La00} for further details on Theorem \ref{thm: Ore embedding}. 
Not all group-embeddable semigroups are Ore: 
for example, see \cite{Pa02a} for the case of Artin monoids. 
It is also a pertinent question as to which group we choose to embed $P$ into, if there are multiple options. 
Here, we may carry out a universal construction following \cite{CP67}. 
If a semigroup $P$ is group-embeddable, there exists a universal group embedding $\iota_{univ} : P \hookrightarrow G_{univ}(P)$ such that for any other group embedding $\iota : P \hookrightarrow G$, we have a unique homomorphism $\Phi:G_{univ}(P) \to G$ such that $\iota = \Phi \circ \iota_{univ}$, that is, the following diagram commutes:
\begin{center}
    \begin{tikzcd}
P \arrow[r, hook, "\iota_{univ}"] \arrow[rd, hook, "\iota"] & G_{univ}(P) \arrow[d, "\Phi"] \\
& G
\end{tikzcd}
\end{center}
Hence, 
it will suffice for us to just state that that $P$ embeds into some group $G$ without specifying which group it is. 

In the group-embeddable setting, A. Nica first defined quasi-lattice ordered semigroups in \cite{Ni92}, and studied their isometric covariant representations. 
The following generalisation, which was developed in \cite{BRRW14}, removed the group-embeddability assumption and the reference to the order. 
\begin{defn}\label{defn: right LCM}
    A left-cancellative semigroup $P$ is said to be right LCM if for all $p,q \in P$, either $pP \cap qP = rP$ for some $r \in P$, or $pP \cap qP = \emptyset$. 
\end{defn}
This $r$ may possibly be non-unique, but one observes directly that if $r$ and $r'$ both satisfy $pP \cap qP = rP = r'P$, then there exists $u \in P^\star$ such that $r' = r u$. 
We denote the set of right LCM elements of $p,q$ by $p \vee q = \setof{r \in P \st  pP \cap qP = rP}$. 
For a finite subset $F \subset P$, we denote $\vee F := \setof{r \in P \st \bigcap_{p \in F} pP = rP}$. 
These semigroups have been referred to and studied under various names in the past: for examples, M. Lawson considers the dual definition in \cite{La99} as CRM monoids (after Clifford, Reilly, and McAlister), and M. D. Norling refers to them as semigroups satisfying Clifford's condition in \cite{No14}. 

We clarify a further bit of notation for the subsequent considerations. 
For all $p \in P$ and $X \subseteq P$, we denote $p^{-1}X := \setof{x \in P \; \vert \; px \in X}$. 
Further, if $P$ embeds in $G$, $p^{-1}X = \setof{p^{-1}x \; \vert \; x \in G} \cap P$. 
For $p_1,q_1, \dots , p_n, q_n \in P$, the set $p_1^{-1}q_1p_2^{-1}q_2 \dots p_n^{-1}q_nP$ does not depend on the actual product of these elements, but is equal to $p_1^{-1}(q_1(p_2^{-1}(q_2( \dots (p_n^{-1}(q_nP))))))
$, which is equal to 
$p_1^{-1}q_1P \cap p_1^{-1}q_1p_2^{-1}q_2P \cap \dots \cap p_1^{-1}q_1p_2^{-1}q_2 \dots p_n^{-1}q_nP$. 
Finally, we conclude this section with some examples:
\begin{ex}
    We look at the following classes of examples:
    \begin{enumerate}[(i)]    
        \item For any $k \in \bN$, the free abelian semigroup on $k$ generators, 
        ${\bN^k} = {\mathbb{Z}^k}^+ = \left\langle a_1, \dots a_k \; \vert \; a_i a_j = a_j a_i \;\forall \; i,j\right\rangle^+$ is Ore, 
        and embeds into $\Z^k$. 
        It is right LCM, with the unique LCM given by $(m_1, \dots, m_k) \vee (n_1, \dots, n_k) = (\max\{m_1, n_1\}, \dots, \max\{m_k, n_k \})$ for $m_i,n_i \in \bN$. 
        (We shall put a $+$-sign in the superscript to distinguish the semigroup and group representations from now on.)
        
        \item For any $k \in \bN$, the free semigroup on $k$ generators $\bF_k^+ = \left\langle a_1, \dots a_k \right\rangle^+$ is not 
        Ore, as it is 
        not 
        reversible.
        However, 
        $\bF_k^+$ canonically embeds into the free group on $k$ generators, $\bF_k = \left\langle a_1, \dots a_k \right\rangle$, by mapping the generators to their counterparts.
        It is also right LCM. 
        For two semigroup elements $p$ and $q$, if there exists an $r \in \mathbb{F}_k^+$ such that $p = qr$ (resp. $q = pr$), then $p \vee q = p$, (resp. $p \vee q = q$) otherwise $p \vee q = \emptyset$.
        
        
        \item Artin monoids are an important case. 
        Denote by $\langle s,t \rangle_m$ the word where $s$ and $t$ alternate $m$ times, that is, $\langle s,t \rangle_4 = stst$. 
        Consider a symmetric matrix $M=(m_{i,j})_{i,j}$ for $1\leq i,j \leq n$ such that $m_{i,i}= 1$ for all $i$, and $m_{i,j} \in \setof{2,3,\dots,+\infty}$ for $i\neq j$. 
        Then, define $A_M^+ := \left\langle e_1, \dots, e_n \; \vert \; \langle e_i, e_j\rangle_{m_{i,j}} = \langle e_j, e_i\rangle_{m_{j,i}}\right\rangle^+$ to be the Artin monoid associated with $M$, and $A_M := \left\langle e_1, \dots, e_n \; \vert \; \langle e_i, e_j\rangle_{m_{i,j}} = \langle e_j, e_i\rangle_{m_{j,i}}\right\rangle$ to be the Artin group associated with $M$, with the convention that $m_{i,j}=\infty$ implies that there is no relation between $e_i$ and $e_j$. 
        $A_M^+$ embeds into $A_M$ \cite{Pa02a}, but is not always Ore. 
        $A_M^+$ is always right LCM. 
%
If $m_{i,j} \in \setof{2, +\infty}$ for $i \neq j$, we call $A_M^+$ (resp. $A_M$) a right-angled Artin monoid (resp. right-angled Artin group). 
        Right-angled Artin groups are very closely related to Coxeter groups $C_M$, generated by the same relations as $A_M$ along with the additional relation that $e_i^2 = e$. 
        If $C_M$ is finite, then we say that $A_M^+$ is of finite type (or spherical type). 
        Their $C^\star$-algebras have been discussed in detail in \cite{CL02, CL07}. 
        It was proved in \cite{De72} that $A_M^+$ is Ore if and only if it is of finite type. 
        Note that if $A_M^+$ has $k$ generators with $m_{i,j} = 2$ for all $i \neq j$ then we get $\bN^k$, while for $m_{i,j} = \infty$ we get $\bF_k^+$. 
        
    \end{enumerate}
\end{ex}

\section{Semigroup $C^\star$-algebras and Amenability}

In recent years, there has been significant interest in the study of $C^\star$-algebras and operator algebras of semigroups \cite{CELY17}. 
The representation theory of semigroups is much more nuanced than their counterpart for groups. 
While it is a perfectly good idea to define the reduced semigroup $C^\star$-algebra of a semigroup $P$ as the norm-closed subalgebra of $\B(\ell_2(P))$ generated by the (isometric) left-regular representation and closed under involution, unfortunately, considering the universal algebra generated by these isometries yields something that is far too large. 
In fact, it was proved by G. J. Murphy \cite{Mur96} that such a $C^\star$-algebra for even $\bN^2$, the universal $C^\star$-algebra generated by two commuting isometries, is non-nuclear! 
The problem was addressed by X. Li in \cite{Li12a} by insisting that the representation keep track of the ideal structure of the semigroup as well, and this led to proper notions of universal $C^\star$-algebras of semigroups. 
X. Li also studied the notion of amenability in the context of semigroups in the same paper. 

\subsection{Semigroup $C^\star$-algebras}

Consider a discrete countable left cancellative semigroup $P$, and its left-regular representation $V: P \to \B(\ell_2(P))$, $V_p(\delta_q) = \delta_{pq}$. 
It is straightforward to verify that $V$ is an isometric representation (i.e., $V_p$ is an isometry for every $p \in P$). 
This allows us to define the following reduced semigroup algebras. 
\begin{defn}
    We define the \textit{reduced semigroup $C^\star$-algebra} of $P$ to be the $C^\star$-algebra generated by its left-regular representation, i.e. $C^\star_\lambda(P) := C^\star(V_p \; \vert \; p \in P) \subseteq \B(\ell_2(P))$. 
    Similarly, we define the \textit{reduced semigroup operator algebra} of $P$ to be the (possibly non-selfadjoint) operator algebra generated by the left-regular representation, i.e. $\A_\lambda(P) := \overline{\textrm{alg}}\left\{V_p \; \vert \; p \in P\right\} \subseteq \B(\ell_2(P))$. 
\end{defn}
Some authors also refer to $C^\star_\lambda(P)$ as the Toeplitz algebra of the semigroup, $\mathcal{T}_\lambda(P)$. 
%
A proper 
candidate for a universal semigroup $C^\star$-algebra was proposed by X. Li in \cite{Li12a}, using the structure of the space of right ideals of the semigroup.
We say a subset $X \subseteq P$ is a right (resp. left) ideal of $P$ if $Xp \subseteq X$ (resp. $pX \subseteq X$) for all $p \in P$. 
Groups lack non-trivial ideals, but for semigroups such spaces can be interesting. 
We further restrict ourselves to the space of \textit{constructible right ideals}, which is defined in the following fashion:

\begin{defn}
    We define the set of \textit{constructible right ideals} $\J$ of $P$ to be the smallest family of right ideals of $P$ satisfying the following properties:
    \begin{enumerate}
        \item[(I1)] $P \in \J$ and $\emptyset \in \J$. 
        \item[(I2)] For all $p \in P$ and $X \in \J$, $pX \in \J$ and $p^{-1}X \in \J$. 
        \item[(I3)] For all $X,Y \in \J$, $X \cap Y \in \J$. 
    \end{enumerate}
\end{defn}

One might observe here that the set $\J$ is actually a semilattice, i.e. an abelian semigroup of idempotents with the operation being intersection of ideals. 
There is a duality between semilattices and 
totally disconnected locally compact Hausdorff spaces. 
We refer the reader to 
\cite{CELY17} for further details on this duality. 
It is proved in Lemma 3.3 of \cite{Li12a} that (I1) and (I2) imply (I3), that is, one may state:
\begin{align}\label{eq: form of Ideals}
    \J = \left\{ q_1^{-1} p_1 \dots q_n^{-1} p_n P \; \vert \; n \in \mathbb{N}, \textrm{ and }p_i,q_i \in P \textrm{ for all } i\in\{1, \dots n\} \right\} \cup \{\emptyset\}.
\end{align}
We further need to define the following independence condition on $\J$, from \cite{Li12a}:
\begin{defn}\label{defn: independence condition}
    If for all $n \in \mathbb{N}$ and all collections $\{X_i\}_{i=1}^n$ of ideals in $\J$, $X = \cup_{i=1}^n X_i \in \J$ implies $X=X_i$ for some $i$, then we say that $\J$ is \textit{independent}. 
\end{defn}

    If $P$ is right LCM (as in Definition \ref{defn: right LCM}), then 
    all constructible right ideals are \textit{principal}, that is, $\J = \{ pP \; \vert \; p \in P \}$. 
    The independence condition on $\J$ is thus automatically satisfied if $P$ is right LCM. 
%
In \cite{Li12a}, X. Li formulated the following two universal $C^\star$-algebras, Definitions \ref{defn: full semigroup C*-algebra} and \ref{defn: constructible semigroup C*-algebra}, which capture the ideal structure of the semigroup:

\begin{defn}\label{defn: full semigroup C*-algebra}
    The \textit{full semigroup $C^\star$-algebra} $C^\star(P)$ of a semigroup $P$ is the universal $C^\star$-algebra constructed by the set of isometries $\{v_p \; \vert \; p \in P\}$ and the set of projections $e_X \; \vert \; X \in \J\}$ satisfying the following relations, for all $p,q \in P$ and all $X,Y \in \J$:
    \begin{enumerate}
        \item[(F1i)] $v_p v_q = v_{pq}$, 
        \item[(F1ii)] $v_p e_X v_p^\star = e_{pX}$,
        \item[(F2i)] $e_P = 1$,
        \item[(F2ii)] $e_\emptyset = 0$, and 
        \item[(F2iii)] $e_X e_Y = e_{X \cap Y}$. 
    \end{enumerate}
\end{defn}

By Lemma 2.8 of \cite{Li12a}, 
these definitions imply 
that $v_p^\star e_X v_p = e_{p^{-1}X}$, and thus 
\begin{align}\label{eq: e-v relations}
    e_{q_1^{-1} p_1 \dots q_n^{-1} p_n P} = v_{q_1}^\star v_{p_1} \dots v_{q_n}^\star v_{p_n} v_{p_n}^\star v_{q_n} \dots v_{p_1}^\star v_{q_1}
\end{align}

By equation \ref{eq: form of Ideals}, this proves that $C^\star(P)$ is generated by the isometries $\{v_p \; \vert \; p \in P\}$. 
The canonical surjective $\star$-homomorphism, the left-regular representation $\lambda:C^\star(P) \to C^\star_\lambda(P)$, can thus be affected by mapping $v_p$ to $V_p$ for every $p \in P$. 
We define the diagonal subalgebra $$D(P) := C^\star(e_X \; \vert \; X \in \J) \subseteq C^\star(P),$$
that is, the universal $C^\star$-algebra generated by the projections corresponding to elements in $\J$. 
By (F2iii), one may observe that this is a commutative $C^\star$-algebra. 
For every $X \in \J$, one has the projection $E_X \in C^\star_\lambda(P)$ (on the subspace $\ell_2(X) \subseteq \ell_2(P)$) which may be obtained from equation \ref{eq: e-v relations} by applying the map $\lambda$. 
This allows one to define $$D_\lambda(P) := C^\star(E_X \; \vert \; X \in \J) = \lambda(D(P)) \subseteq \B(\ell_2(P)).$$

X. Li further defines the following universal $C^\star$-algebra when $P$ is group-embeddable, which is instrumental in most considerations:
\begin{defn}\label{defn: constructible semigroup C*-algebra}
     Let $P$ be a semigroup which embeds into a group $G$. 
     The \textit{constructible semigroup $C^\star$-algebra} $C^\star_s(P)$ of a semigroup $P$ is the universal $C^\star$-algebra constructed by the set of isometries $\{v_p \; \vert \; p \in P\}$ and the set of projections $\{e_X \; \vert \; X \in \J\}$ satisfying the following relations, for all $p,q \in P$ and all $X,Y \in \J$:
     \begin{enumerate}
         \item[(S1)] $v_p v_q = v_{pq}$, 
         \item[(S2)] $e_\emptyset = 0$, 
         \item[(S3)] for all $p_1, q_1, \dots p_n, q_n \in P$ satisfying $p_1^{-1} q_1 \dots p_n^{-1} q_n = e$ in $G$, 
         $$v_{p_1}^\star v_{q_1} \dots v_{p_n}^\star v_{q_n} = e_{q_n^{-1} p_n \dots q_1^{-1} p_1 P}. $$
     \end{enumerate}
\end{defn}

It can be shown that conditions (F1ii), (F2i), (F2iii) can be deduced to hold from (S3), and hence there exists a canonical surjective $\star$-homorphism $\pi_s:C^\star(P) \to C^\star_s(P)$ that maps $v_p\in C^\star(P)$ to $v_p\in C^\star_s(P)$ and $e_X\in C^\star(P)$ to $e_X\in C^\star_s(P)$. 
$C^\star_s(P)$ is also generated by $\{v_p \; \vert \; p \in P\}$, and one has a a canonical $\star$-homomorphism, the left-regular representation, $\lambda_s:C^\star_s(P) \to C^\star_\lambda(P)$, which maps $v_p$ to $V_p$. 
Furthermore, one defines $$D_s(P) := C^\star(e_X \; \vert \; X \in \J) \subseteq C^\star_s(P),$$
and $\pi_s\vert_{D(P)}:D(P) \to D_s{(P)}$ is a surjective $\star$-homomorphism. 

That is, one has the following commutative diagrams:

\begin{center}
    \begin{tikzcd}
C^\star(P) \arrow[rd, twoheadrightarrow, "\lambda"] \arrow[d, twoheadrightarrow, "\pi_s"] \\
C^\star_s(P) \arrow[r, twoheadrightarrow, "\lambda_s"] & C^\star_\lambda(P)
\end{tikzcd}
and 
\begin{tikzcd}
D(P) \arrow[rd, twoheadrightarrow, "\lambda\vert_{D(p)}"] \arrow[d, twoheadrightarrow, "\pi_s\vert_{D(p)}"] \\
D_s(P) \arrow[r, twoheadrightarrow, "\lambda_s\vert_{D_s(p)}"] & D_\lambda(P)
\end{tikzcd}
\end{center}

\begin{cor}[Corollary 3.3 of \cite{KKLL22}]
    The following are equivalent:
    \begin{enumerate}
        \item[(i)] $\J$ is independent. 
        \item[(ii)] $\lambda\vert_{D(P)}$ is an isomorphism. 
        \item[(iii)] $\lambda\vert_{D_s(P)}$ is an isomorphism. 
    \end{enumerate}
\end{cor}


One should be careful that unlike the reduced case, the term ``universal Toeplitz $C^\star$-algebra'' $\mathcal{T}_u(P)$ means something different from the universal $C^\star$-algebras defined here. 
M. Laca and C. Sehnem defined and studied these universal algebras. 
$\mathcal{T}_u(P)$ is in general a quotient of $C^\star_s(P)$, and they are isomorphic if and only if $\J$ satisfies the independence condition. 
We refer the reader to \cite{LS22} for details. 

\begin{rem}
    It is not completely clear when the map $\pi_s:C^\star(P) \to C^\star_s(P)$ will be an isomorphism. 
    It can be shown to be true when $P$ is quasi-lattice ordered, or when $P$ is Ore. 
    However, as we shall see in the next subsection, in the amenable case, if $\J$ is independent, $C^\star(P) \cong C^\star_s(P) \cong C^\star_\lambda(P)$. 
    Moreover, these $C^\star$-algebras are nuclear. 
\end{rem}

\subsection{Amenability of Semigroups}

We briefly look at how amenability of the semigroup manifests itself in the relation between its semigroup $C^\star$-algebras. 
We state the definition used by X. Li in \cite{Li12a}:
\begin{defn}\label{defn: left-amenable semigroup}
    Let $P$ be a discrete semigroup. 
    We say $P$ is left-amenable if there exists a left-invariant mean on $\ell_\infty(P)$, that is, a linear functional $\mu\in (\ell_\infty(P))^\star$ such that $\mu(f(p\sqcup)) = \mu(f)$ for all $f\in\ell_\infty(P)$ and all $p \in P$, where $f(p\sqcup)$ refers to the composition of $f$ with the left-multiplication by $p$. 
\end{defn}

It is also possible to define right-amenability in a similar fashion. 
Amenability of semigroups is in general a broad and interesting topic: for example, X. Li showed in \cite{Li12a} that one may obtain for semigroups conditions analogous to (1)-(7), Theorem 6.8, \cite{BO08}, as stated for groups. 
However, we restrict ourselves to only two of these conditions for the moment, and refer the interested reader to Section 4 of \cite{Li12a}, or to A. Paterson's celebrated monograph \cite{Pa88} for further details. 
\begin{thm}[Part of Statements 4.1, \cite{Li12a}]
    Let $P$ be a group-embeddable semigroup, consider the following two conditions:
    \begin{enumerate}
        \item[(i)] $P$ is left-amenable.
        \item[(ii)] The left-regular representation $\lambda:C^\star_s(P) \to C^\star_\lambda(P)$ is an isomorphism and there exists a non-zero character on $C^\star_s(P)$. 
    \end{enumerate}
    It is always true that $(i)\implies (ii)$. 
    In addition, if $\J$ is independent, then $(ii) \implies (i)$ as well. 
\end{thm}
One further has, for the full semigroup $C^\star$-algebra:
\begin{thm}[Part of Theorem 5.6.42, \cite{CELY17}]\label{thm: lambda isomorphism from cancellative}
    Let $P$ be a cancellative semigroup, and $\J$ be independent. 
    The following are equivalent:
    \begin{enumerate}
        \item[(i)] $P$ is left-amenable.
        \item[(ii)] The left-regular representation $\lambda:C^\star(P) \to C^\star_\lambda(P)$ is an isomorphism and there exists a non-zero character on $C^\star(P)$. 
    \end{enumerate}
\end{thm}



A. Nica developed the following notion of amenability for quasi-lattice ordered semigroups (generalisable to the right LCM case) in \cite{Ni92}: 
\begin{defn}\label{defn: Nica-amenable semigroup}
    A right LCM semigroup $P$ is said to be Nica-amenable if the left-regular representation $\lambda:C^\star(P) \to C^\star_\lambda(P)$ is an isomorphism. 
\end{defn}

\begin{rem}
    For right LCM semigroups, the independence condition of $\J$ always holds, and hence conditions (i) and (ii) of Theorem \ref{thm: lambda isomorphism from cancellative} are equivalent. 
    At this point, it is tempting to conclude that Nica-amenability as defined in Definition \ref{defn: Nica-amenable semigroup} is equivalent to the general notion of left-amenability as defined in Definition \ref{defn: left-amenable semigroup}. 
    However, that is not actually true, since even if $P$ is Nica-amenable, it might fail to have a non-zero character on the $C^\star$-algebra. 
    And this is indeed what happens for $\mathbb{F}_k^+$, as we shall see in the examples below. 
    So, for right LCM semigroups, left-amenability in the sense of X. Li is strictly stronger than Nica-amenability. 
    In the subsequent considerations, we shall use ``amenability'' to mean left-amenability in the sense of X. Li, unless explicitly mentioned otherwise. 
\end{rem}

We conclude this section with a couple of examples (and non-examples):
\begin{ex}
    We return to the examples listed at the end of Section 2:
    \begin{enumerate}
        \item[(i)] ${\mathbb{N}^k}$ is amenable
        . 
        In fact, all abelian semigroups are left (and right) amenable, both in the sense of X. Li and A. Nica. 

        \item[(ii)] 
        $\mathbb F_k^+$ is 
        Nica amenable, but not amenable in the sense of X. Li. 
        We refer the reader to \cite{Ni92} for a proof. 


        \item[(iii)] 
        Artin monoids 
        are Nica amenable if and only if they are right-angled \cite{CL02, LL20}. 

    \end{enumerate}
\end{ex}

\section{Boundary Quotients}

Early studies into boundary quotients of semigroups were carried out by J. Crisp and M. Laca \cite{CL07} for right-angled Artin monoids ($\bF_k^+$ and $\bN^k$ are both special cases of these), and they also investigated questions of pure infiniteness and simplicity. 
Their construction was generalised in two separate ways, which we briefly describe in this section. 
If $P$ embeds into a group $G$, there is a natural partial action of $G$ on the (abelian) reduced diagonal subalgebra $D_\lambda(P)$, and consequently its spectrum $\Omega_P$. 
$\Omega_P$ can be shown to have a unique smallest $G$-invariant subspace
, which we refer to as $\partial \Omega_P$. 
Then, the reduced crossed product $C^\star$-algebra $\partial C^\star_\lambda(P) := C(\partial\Omega_P) \rtimes_rG$ is defined to be the boundary quotient $C^\star$-algebra of the semigroup
. 
It satisfies certain desirable co-universal properties: 
it is the co-universal object in the category of $G$-equivariant representations of $C^\star_s(P)$, and is isomorphic to the $C^\star$-envelope of the co-system of $G$ acting on $\A_\lambda(P)$. 
We refer the reader to \cite{Se19, LS22, DKELL22, KKLL22} for details. 
For right LCM semigroups $P$, 
one also has a characterisation of the boundary quotient $\mathcal{Q}(P)$ via a maximality of relations picture.  
One defines foundation sets of $P$ as those which generate principal ideals that intersect any other principal ideal non-trivially. 
A vanishing condition on the products of complement projections for the ideals corresponding to the foundation set elements, generates an ideal of the universal $C^\star$-algebra, and quotienting with respect to this ideal yields the boundary quotient $C^\star$-algebra. We refer the reader to \cite{St15, BRRW14, BS16, LL22} for details. 
The connection between them has been briefly discussed in \cite{LS22}. 
The core of the idea is that closed $G$-invariant subsets of the Nica spectrum ($\Omega_P$) correspond to sets of elementary relations on $P$, and the unique smallest $G$-invariant subset $\partial\Omega_P$ corresponds to the maximal such set of elementary relations. 

\subsection{Boundary Quotients via Partial Crossed Products}
One way to define boundary quotients in the general setting is through partial actions, which are a generalisation of the notion of an action. 
The core of the idea is that instead of an action of a group by bijections of a set (or suitable isomorphisms, e.g. homeomorphisms in the topological case), we only ask for \textit{partial} bijections on a set, that is, maps which are bijections on suitably restricted domains, and compositions / inverses that are defined on suitably restricted domains as well. 
It is also possible to define reduced and universal (partial) crossed product $C^\star$-algebras from these actions, and one has the canonical surjective $\star$-homomorphism from the latter to the former. 
We refer the interested reader to \cite{Ex17, Li17b} for further background. 


Now, we consider the inverse semigroup for $P$, constructed in the following fashion using the left-regular representation:
$$\I_V = \left\{ V_{p_1}^\star V_{q_1} \dots V_{p_n}^\star V_{q_n} \;\vert\; \textrm{ for all } p_i,q_i \in P, \; 1\leq i \leq n, \; n \in \mathbb{N} \right\}.$$
It is straightforward to verify that it is indeed an inverse semigroup (recall that an inverse semigroup is a semigroup satisfying that for every $x$ there exists a unique $y$ with $xyx=x$) with $VV^\star V =V$ for all $V \in \I_V$. 

Assume that $P$ embeds into a group $G$. 
We denote $\I_V^\times = \I_V \backslash\{0\}$, and define the homomorphism $$\sigma:\I_V^\times \to G, \;\;\;\;\;\;\;\;\;\;V_{p_1}^\star V_{q_1} \dots V_{p_n}^\star V_{q_n} \mapsto p_1^{-1} q_1 \dots p_n^{-1} q_n.$$ 

It can be proved \cite{KKLL22} that $D_\lambda(P) = \overline{span}\{\sigma^{-1}(e)\}$. 
We further define 
$$D_{g^{-1}} := \overline{span}\{V^\star V \;\vert\; V \in \I_V^\times, \; \sigma(V) = g\} \subseteq D_\lambda(P).$$ 
It was shown in Section 3.3 \cite{Li17b}, that $D_{g^{-1}}$ is an ideal of $D_\lambda(P)$ for every $g \in G$. 
Noting that $\ell_2(P)$ is a subspace of $\ell_2(G)$, we may identify $C^\star_\lambda(P) \subseteq \B(\ell_2(P))$. 
We observe that for all $V \in \sigma^{-1}(g)$, denoting $\lambda_g$ as the left-regular representation of $G$ on $\ell_2(G)$, $V = \lambda_g V^\star V$, and hence it induces a $\star$-homomorphism 
$$\theta_g^\star: D_{g^{-1}} \to D_g, \;\;\;\;\;\;\;\;\;\; V^\star V \mapsto \lambda_g V^\star V \lambda_g^\star = VV^\star .$$
By construction, $D_\lambda(P)$ and $D_{g^{-1}}$ are abelian $C^\star$-algebras, and if we denote their spectra as $\Omega_P$ and $\Omega_{g^{-1}}$ respectively, we may identify $\Omega_{g^{-1}}$ as an open subset of $\Omega_P$, and $\theta_g^\star$ as described above induces $\theta_{g^{-1}}$ to be a partial action of $G$ on $\Omega_P$ 
by 
$$\theta_g^\star(\chi) = \chi \circ \theta_{g^{-1}}, \;\;\;\;\textrm{ for all } \chi \in D_{g^{-1}} \cong C_0(\Omega_{g^{-1}}).$$

\begin{prop}[Proposition 3.12, \cite{KKLL22}; Proposition 3.10, \cite{Li17b}]\label{prop: partial crossed product realisation of the reduced semigroup C* algebra}
    Let $P$ be a semigroup-embeddable in a group $G$. 
    There exists a canonical isomorphism 
    $$ C^\star_\lambda(P) \to D_\lambda(P) \rtimes_r G
    ,$$
    which maps $V_p$ to the image of $p \in G$ in $D_\lambda(P) \rtimes_r G$. 
\end{prop}
In \cite{Li17b} it was shown that $\Omega_P$ has a unique smallest $G$-invariant subset, which we shall denote by $\partial \Omega_P$. 
The space $\partial \Omega_P$ can be identified using the semilattice of idempotents on $P$. 

\begin{defn}\label{defn: reduced boundary quotient}
    Let $P$ be a semigroup that embeds in a group $G$, and the space $\partial\Omega_P$ defined as above. 
    The boundary quotient of $C^\star_\lambda(P)$ is defined to be:
    $$\partial C^\star_\lambda(P) := C(\partial\Omega_P) \rtimes_r G.$$
\end{defn}

One may define full boundary quotients as well, in an analogous fashion as done above, we refer the interested reader to \cite{LS22} for further details. 
And now we conclude this subsection by stating a co-universality criterion for the boundary quotient. 
\begin{defn}\label{defn: constructible representation}
    We call an isometric representation of $P$ \textit{constructible} if it induces a representation of $C^\star_s(P)$. 
\end{defn}
 
\begin{thm}[Theorem 4.2, \cite{KKLL22}]\label{thm: co-universality of boundary quotient}
    Let $P$ be a semigroup, embeddable in a group $G$. 
    If $\partial_G[C^\star_s(P)]$ is the $C^\star$-algebra of a $G$-equivariant constructible representation of $P$ such that for every $G$-equivariant constructible representation $T$, there exists a canonical surjective $\star$-homomorphism $C^\star(T) \to \partial_G[C^\star_s(P)]$ that maps the generators to themselves, then 
    $$\partial C^\star_\lambda(P) \cong \partial_G[C^\star_s(P)].$$
\end{thm}

In the right LCM cancellative setting situation, even if $P$ is not group-embeddable, it is still possible to construct boundary quotients through normal coactions, as developed in \cite{BBD26}.

\subsection{Boundary Quotients via Foundation Sets}

One has an alternate definition of the boundary quotient in terms of foundation sets. 
This definition was developed in the context of right LCM semigroups by \cite{St15, BRRW14, LL22}, and extended to the general setting in \cite{LS22}, and further refined in the right LCM setting in \cite{BS16}. 
C. Starling proved in \cite{St15} that this boundary quotient is in fact the tight $C^\star$-algebra of the inverse semigroup, and addressed the issues of pure infiniteness and simplicity. 

\begin{defn}\label{defn: boundary quotient in terms of foundation sets}
    Let $P$ be a right LCM semigroup. 
    A finite subset $F$ is called a \textit{foundation set} for $P$ is for any $p \in P$, there exists some $f \in F$ such that $pP \cap fP \neq \emptyset$. 
%
    The boundary quotient $\mathcal{Q}(P)$ is the universal $C^\star$-algebra generated by the isometries $\{v_p \;\vert\; p \in P\}$ and the projections $\{e_X \;\vert\; X \in \J\}$ satisfying the relations in Defintion \ref{defn: full semigroup C*-algebra} and that whenever $F$ is a foundation set of $P$,
    \begin{align}\label{eqn: foundation set}
        \Pi_{f \in F} (1-e_{fP}) = 0.
    \end{align}
\end{defn}
That is, $\mathcal{Q}(P)$ is a quotient of $C^\star(P)$ by the ideal generated by relations of the form $\Pi_{f \in F}(1-e_{fP}) = 0$ for every foundation set $F$. 
As pointed out by B. Li and M. Laca in \cite{LL22}, denoting $I_\infty$ to be the ideal of $D(P) \cong D_\lambda(P) = C(\Omega_P)$ generated by projections of the form $\Pi_{f \in F} (1-e_{fP}) = 0$ for every foundation set $F$, one has $C(\partial \Omega_P) \cong C(\Omega_P) / I_\infty$, and $\mathcal{Q}(P) \cong C(\partial \Omega_P) \rtimes G \cong C(\partial \Omega_P) \rtimes P$, where the last equivalence (partial crossed product with respect to a semigroup action) is defined and studied in \cite{LL22}. 

This notion may extended beyond the right LCM setting as well, as noted in Remark 5.5 of \cite{BRRW14} by replacing finite subsets of $P$ by finite subsets of $\J$ which have non-trivial intersection with every ideal (since not every constructible ideal is principle in the general case). 
This was further studied in \cite{LS22}, where they 
used it in the concrete presentation of the covariance algebra, which is the full counterpart of the crossed-product version of the boundary quotient. 
%
We conclude this section by recognising that in practice it might be difficult to take the quotient with respect to \textit{every} foundation set, but N. Brownlowe and N. Stammeier developed the notion of \textit{accurate refinement property} in \cite{BS16}, where they reduced the computation by restricting to only certain types of foundation sets, and showed that this property holds in practice for a very large class of semigroups, in particular including directed semigroups and semigroups where principal right ideals corresponding to non-comparable elements are disjoint.

\subsection{Computing Boundary Quotients}


We begin this section with the following celebrated result by C. Sehnem. 
\begin{thm}[Corollary 5.4, \cite{Se22} and Theorem 5.4, \cite{BBD26}]\label{thm: boundary quotient is isomorphic to C*-envelope}
    Let $P$ be a semigroup that is either group-embeddable, or is right LCM cancellative. 
    Then 
    $$\partial C^\star_\lambda(P) \cong C^\star_{env}(\A_\lambda(P)),$$
    via an isomorphism that identifies the canonical generating isometries. 
\end{thm}
Furthermore, when $P$ is right-Ore (recall Definition \ref{defn: Ore}: cancellative and left-reversible), then $G=PP^{-1}$, $\partial\Omega_P$ is just a single point, and thus $\partial C^\star_\lambda(P) = C(\partial\Omega_P) \rtimes_r G = C^\star_\lambda (G)$. 
This brings us to the following theorem: 
\begin{thm}[Part of Theorem 4.6, \cite{KKLL22}]\label{thm: boundary quotients for right-Ore semigroups}
    Let $P$ be a semigroup that embeds in a group, and let $G$ be the group that it generates. 
    Then, the following are equivalent:
    \begin{enumerate}
        \item[(i)] $P$ is right-Ore, which implies that $G = PP^{-1}$ (by Theorem \ref{thm: Ore embedding}).
        \item[(ii)] The map $V_p \mapsto \lambda_p$ extends to a completely isometric map $\A_\lambda(P) \to C^\star_\lambda(G)$. 
        \item[(iii)] $C^\star_{env}(\A_\lambda(P)) \cong  C^\star_\lambda(G)$ by a canonical $\star$-homomorphism that fixes $P$. 
    \end{enumerate}
    Any of the conditions above further imply that $\A_\lambda(P)$ is hyperrigid. 
\end{thm}


We state and prove the following proposition about $\mathcal{Q}(P)$ below . 
\begin{prop}\label{prop: Q(P) for left reversible}
    If $P$ group-embeddable, right LCM, and left-reversible, (i.e. right-Ore) then
    $$\mathcal{Q}(P) \cong C^\star(G)
    ,$$ 
    where $G$ is the group $P$ generates. 
\end{prop}

\begin{proof}
    If $P$ is left-reversible (Definition \ref{defn: Ore}), then for every $f,p \in P$, $pP\cap qP \neq \emptyset$. 
    This implies that 
    every singleton $\{p\}\subset P$ 
    is a foundation set. 
    Hence, if the condition \ref{eqn: foundation set} in the definition of boundary sets reduces to 
    $$1-v_pv_p^\star = 0 \;\; \textrm{ for all } \;\; p \in P.$$
    That is every isometry $v_p$ generating $\mathcal{Q}(P)$ is a co-isometry as well, hence a unitary. 
    This implies that the $C^\star$-algebra generated by these unitaries is just the (universal) $C^\star$-algebra of the group it generates. 
    This proves that $\mathcal{Q}(P) \cong C^\star(G)$. 
\end{proof}

And now, we provide some explicit examples of computations below. 


\begin{ex}
    We compute the boundary quotients for the semigroups we have considered so far:
    \begin{enumerate}
        \item[(i)] For $\mathbb{N}^k$, since it is Ore, and $\Z^k$ is its enveloping group, by Theorem \ref{thm: boundary quotients for right-Ore semigroups}, $\partial C^\star_\lambda(\mathbb{N}^k) \cong C^\star(\Z^k)$. 
        Even from the foundation set viewpoint, we see that since $\Z^k$ is left-reversible, by Proposition \ref{prop: Q(P) for left reversible}, we have that $\mathcal{Q}(P) \cong C^\star(\Z^k)$ as well. 
        \item[(ii)] For $\mathbb{F}_k^+$, 
        $\partial C^\star_\lambda(\mathbb F^+_k) \cong \cO_k$ (the reader is referred to \cite{CELY17} for a direct proof of the $k=2$ case, or an application of Theorem 6.7, \cite{CL07} for the general case). 
        Section 5.1 of \cite{St15} shows that $\mathcal{Q}(\mathbb{F}_k^+) = \mathcal{O}_k$ too. 
        \item[(iii)] J. Crisp and M. Laca computed $\partial C^\star_\lambda(A_M^+)$ for right-angled Artin monoids $A_M^+$ in \cite{CL07}. 
        They used the quasi-lattice order property to characterise the Nica spectrum $\Omega_P$ and proceeded from there. 
        If $A_M^+$ is of finite type, then it is Ore, and we may use Proposition \ref{prop: Q(P) for left reversible} to obtain $\mathcal{Q}(A_M^+) = C^\star(A_M)$. 
    \end{enumerate}
\end{ex}

\begin{rem}
    We see that $\mathcal{Q}(P) \cong \partial C_\lambda^\star(P)$ in the examples computed above in $(i)$ and $(ii)$. 
    One should not expect them to be isomorphic in general: for example, one might quite trivially consider $P$ to be a non-amenable group $G$, and then $\mathcal{Q}(P) \cong C^\star(G) \ncong C^\star_\lambda(G) \cong \partial C^\star_\lambda(P)$. 
    This happens more generally for right-Ore monoids whose enveloping groups are non-amenable. 
    Connections between the two notions of boundary quotients have been investigated, but to the best of our knowledge, 
    how they are related is not very clear in the most general cases. 
    An interesting problem would be to investigate if there exists a semigroup where the universal counterpart of $\partial C^\star_\lambda(P)$, which is $C^\star_{sc}(P)$ (see \cite{Se19, KKLL22} for details) fails to agree with $\mathcal{Q}(P)$.
    This would imply a non-equivalence of the two notions. 
\end{rem}

\section{Dilating Semigroup Representations to the Boundary}

In this section we address the problem of dilation of a representation of a semigroup to its boundary quotient.  
But first, we briefly discuss dilations of semigroup representations.

\subsection{Dilations of Representations of Semigroups}

Dilations of representations of semigroups are defined in a fashion identical to that for representations of operator algebras. 

\begin{defn}
    Let $P$ be a semigroup, and $T:P \to \B(\cH)$ be a representation of $P$. 
    We say a representation $V: P \to \B(\K)$, with $\K \supseteq \cH$, is a \textit{dilation} of $T$ if $\P_\cH V(p) \vert_\cH = T(p)$ for all $p \in P$. 
    If each element $V(p)$ is an isometry (or unitary), we call $(V, \K)$ an \textit{isometric (corrsp. unitary) dilation}. 
    If $\K = \overline{span}\{ V(p)h \; \vert \; h \in \cH\textrm{, }p \in P \}$, we call the dilation \textit{minimal}. 
\end{defn}

It can be easily seen that if $T$ has an isometric dilation, then $T$ must be a contractive representation (i.e., for each $p\in P$, $T(p)$ is a contraction). 
%
To the best of our knowledge, there is no general characterisation of constructible representations (Definition \ref{defn: constructible representation}). 
However, in the right LCM setting, we do have the following notion of ``covariant representations'', which was first formulated by A. Nica in \cite{Ni92} for quasi-lattice ordered semigroups and their $C^\star$-algebras, and generalised to right LCM semigroups in \cite{CELY17}.

\begin{defn}\label{defn: Nica-covariant}
    Let $P$ be a left-cancellative right LCM semigroup. 
    A representation 
    $V:P \to \B(\cH)$ is said to be an \textit{isometric Nica-covariant  representation} if $V(p)$ is an isometry for every $p$ in $P$, and for all $p, q \in P$,
    \begin{align*}
        V(p) V(p)^\star V(q) V(q)^\star = 
        \begin{cases}
            V(r)V(r)^\star &\textrm{ if } rP = pP \cap qP \textrm{ ,} \\
            0 &\textrm{ if } pP \cap qP = \emptyset \textrm{ .}
        \end{cases}
    \end{align*}
\end{defn}

One may see that this is indeed well-defined: for any $r,r' \in p \vee q$, there exists $u \in P^\star$ such that $r' = ru$, and since $u$ is invertible, so must be $V(u)$, i.e. $V(u)$ must be a unitary. 
Hence $V(r')V(r')^\star = V(ru)V(ru)^\star = V(r) V(u) V(u)^\star V(r)^\star = V(r)V(r)^\star$. 

\begin{ex}\label{ex: isometric Nica-covariant representations}
    We provide a couple of examples of isometric Nica-covariant representations: 
    \begin{enumerate}
        \item[(i)] For $\bN^k$, let $\{e_i\}_{i=1}^k$ be the canonical generating set, and denote $V(e_i) =: V_i$. 
        Then, $V_iV_j=V_jV_i$ necessarily. 
        Every isometric representation is easily seen to be Nica-covariant for $k=1$. 
        For $k \geq 2$, $i\neq j$ implies $e_i \vee e_j = e_i e_j$, and hence Nica-covariance reduces to $V_i V_i^\star V_j V_j^\star = V_i V_j V_j^\star V_i^\star
        $. 
        This reduces to $V_i V_j^\star = V_j^\star V_i$ for all $i \neq j$. 
        That is, the $V_i$'s must form a doubly commuting family of isometries. 
        Any such doubly commuting family of isometries provides a Nica-covariant representation of $\bN^k$. 
    
    \item[(ii)] For $\bF^{+}_k$ 
    with $k \geq 2$, using similar notation as above, $i \neq j$ implies $e_i \vee  e_j = \emptyset$. 
    Hence Nica-covariance demands that $V_i V_i^\star V_j V_j^\star = 0$, which reduces to $V_i^\star V_j = 0$ for all $i \neq j$. 
    Thus, an isometric representation is Nica-covariant if and only if the images of the generators yields a family of isometries with pairwise orthogonal ranges. 
    \end{enumerate}
\end{ex}

One might go a bit further and formulate conditions for contractive representations to admit isometric Nica-covariant dilations. 
We refer the reader to B. Li's work on this \cite{Li19}. 
Motivated by Neumark dilations \cite{Na43, Po99}, B. Li introduced the concept of $\star$-regular dilation in \cite{Li17, Li19}, which we shall not define here. 
We provide a necessary and sufficient condition for the existence of isometric Nica-covariant dilations:
\begin{thm}[Part of Theorem 3.9, \cite{Li19}]\label{thm: Boyu's theorem}
    Let $T$ be a 
    representation of a right LCM semigroup. 
    Then, the following are equivalent:
    \begin{enumerate}[(i)]
        \item $T$ has an isometric Nica-covariant dilation.
        \item For any finite set $F \subset P$, 
        \begin{align}\label{eq: Boyu's condition for isometric Nica-covariant dilation}
            Z(F) = \sum_{U \subseteq F} (-1)^{\lvert U \rvert} T(\vee U) T^\star(\vee U) \geq 0 .
        \end{align}
    \end{enumerate}
\end{thm}

\subsection{Dilating to $\mathcal{Q}(P)$ in the right LCM case}

The following lemma is straightforward, but we still include it for the sake of completeness:

\begin{lem}
    Let $P$ be a right LCM semigroup, and $V: P \to \B(\cH)$ be an isometric Nica-covariant representation of $P$. Furthermore, let $V$ satisfy the condition that for every foundation set $F \subset P$,
    $$\Pi_{f \in F}\left(1_\cH - V(f) V(f)^\star\right) = 0.$$
    Then, there exists a representation of $\mathcal{Q}(P)$, $\psi: \mathcal{Q}(P) \to \B(\cH)$ such that $(\psi\circ q)(v_p) = V(p)$ for all $p$, where $q: C^\star(P) \to \mathcal Q(P)$ is the canonical quotient map.
\end{lem}

\begin{proof}
    The proof of this lemma follows from the universal construction of $\mathcal{Q}(P)$. 
\end{proof}

M. Laca and B. Li further extended this result to contractive representations using Theorem \ref{thm: Boyu's theorem}:

\begin{thm}[Theorem 6.1 and Corollary 6.7, \cite{LL22}]\label{thm: Laca Li dilating to Q(P)}
    Let $P$ be a right LCM semigroup, and $T: P \to \B(\cH)$ be a contractive representation of $P$. Furthermore, let the following two conditions hold:
    \begin{enumerate}
        \item[(i)] For every finite subset $F \subset P$, $Z(F) = \sum_{U \subseteq F} (-1)^{\lvert U \rvert} T(\vee U) T^\star(\vee U) \geq 0$, and
        \item[(ii)] For every foundation set $F \subset P$, $Z(F) = \sum_{U \subseteq F} (-1)^{\lvert U \rvert} T(\vee U) T^\star(\vee U) = 0$
    \end{enumerate}
    Then there exists a dilation $\pi:P \to \B(\K)$ of $T$ to some $\K \supseteq \cH$ such that $\pi$ induces a representation of $\mathcal{Q}(P)$: that is, there exists a representation $\psi: \mathcal{Q}(P) \to \B(\K)$ satisfying
    $$(\psi \circ q)(v_p) = \pi(p)\;\;\;\;\textrm{ for all } p \in P, \textrm{ and }$$
    $$\P_\cH (\psi \circ q) (v_p) \vert_\cH 
    =
    \P_\cH \pi(p)\vert_\cH
    =
    T(p). $$
\end{thm}

We now formulate a dilation theorem for representations of Ore semigroups, that provides a necessary and sufficient condition for unitary dilations. 
\begin{thm}\label{thm: Ore semigroups dilation theorem for full boundary}
    Let $P$ be a right LCM right Ore semigroup, $T: P \to \B(\cH)$ be a representation. 
    Then, the following are equivalent. 
    \begin{enumerate}
        \item[(i)] $T$ has a dilation that induces a representation of $\mathcal Q(P)$. 
        \item[(ii)] $T$ has a unitary dilation. 
    \end{enumerate}
\end{thm}

\begin{proof}
    For $(i) \implies (ii)$, recall that for right Ore semigroups, by Proposition \ref{prop: Q(P) for left reversible} $\mathcal Q(P) \cong C^\star(G)$ by the canonical isomorphism that sends $q(v_p)$ to $u_p$, where $q:C^\star(P) \to \mathcal Q(P)$ is the canonical quotient map, and 
    $u_g$ for $g \in G$ are the (unitary) generators of $C^\star(G)$. 
    Hence, if $\pi: P \to \B(\K)$ is a dilation of $T$ (i.e. $\K \supseteq \cH$ and $T(p) = \P_\cH \pi(p) \vert_\cH$ for all $p \in P$) that induces a representation $\psi: \mathcal Q(P) \to \B(\K)$ satisfying $(\psi \circ q) (v_p) = \pi(p)$, then, as $q(v_p)$ is unitary ($u_p$ being unitary), so is $\pi(p)$. 

    For $(ii) \implies (i)$, consider a unitary dilation $\pi: P \to \B(\K)$ with $\K \supseteq \cH$ and $T(p) = \P_\cH \pi(p)\vert_\cH$ for all $p \in P$. 
    Since $P$ is Ore, $\pi$ extends to a representation of the enveloping group $G$ as well, and consequently a representation of $C^\star(G)$. 
    As $C^\star(G) \cong \mathcal Q(P)$ by Proposition \ref{prop: Q(P) for left reversible}, we have a representation of $\mathcal Q(P)$. 
\end{proof}

And now we 
formulate the counterparts of these results for $\partial C^\star_\lambda(P)$ using the Dritschel-McCullough dilation machine, and show that though our results have a very different flavour, they are stronger in many cases even when the full and reduced boundary quotients agree. 

\subsection{Dilating to $\partial C^\star_\lambda(P)$
}

The following are the main results of our paper. We begin with our generalisation of the matricial von Neumann's inequality. 

\begin{defn}\label{defn: extension property (A)}
    Let $P$ be a semigroup, and $T: P \to \B(\cH)$ be a representation of $P$.
    We say that $T$ has the \textit{extension property (A)} if it extends to a completely contractive representation of $\A_\lambda(P)$. 
    That is, $T:P \to \B(\mathcal{H})$ has extension property (A) iff there exists a completely contractive representation $\phi: \A_\lambda(P) \to \B(\mathcal{H})$ such that $\phi(V_p) = T(p)$ for all $p \in P$. 
\end{defn}

Note that since the algebra $\setof{V_p \; \vert \; p \in P}$ is dense in $\A_\lambda(P)$, there is only one way to extend a representation of $P$ to one of $\A_\lambda(P)$ mapping $V_p$ to $T(p)$. 

\begin{thm}\label{thm: every completely contractive representation of the algebra also extends to the boundary}
    Let $P$ be a semigroup that is group-embeddable or right LCM cancellative, and $\phi: \A_\lambda(P) \to \B(\cH)$ be a 
    representation. 
    Then, the following conditions are equivalent:
    \begin{enumerate}
        \item[(i)] $\phi$ is completely contractive.
        \item[(ii)] $\phi$ has a dilation $\pi: \A_\lambda(P) \to \B(\K)$ of $\phi$ to some $\K \supseteq \cH$ such that $\pi$ canonically induces a representation of $\partial C^\star_\lambda (P)$. 
    That is, we have a representation $\psi: \partial C^\star_\lambda (P) \to \B(\mathcal{K})$ with the property that if $q: C^\star_\lambda(P) \to \partial C^\star_\lambda(P)$ is the canonical quotient map, then 
    $$(\psi \circ q) (V_p) = \pi(V_p) \;\;\;\;\textrm{ for all } p \in P,$$
    where $V_p$ is the left regular representation of $p \in P$. 
    Moreover, since $\pi$ is a dilation of $\phi$, 
    $$\P_\cH (\psi \circ q)(V_p) \vert_\cH 
    =
    \P_\cH \pi(V_p)\vert_\cH
    =
    \phi(V_p)  \;\;\;\;\textrm{ for all } p \in P. $$
    \end{enumerate}
\end{thm}

\begin{proof}
    For the subsequent considerations, we shall abbreviate $\A_\lambda(P)$ as simply $\A_\lambda$ when there is no ambiguity regarding $P$. 
    First we prove $(i) \implies (ii)$. 
    By Theorem \ref{thm: Dritschel-McCoulough theorem}, $\phi$ must have a maximal dilation $\pi: \A_\lambda \to \B(\mathcal{K})$ for some $\mathcal{K} \supseteq \cH$. 
    Since $\pi$ is a maximal representation, by Theorem \ref{thm: Unique Extension Property}, it must also have the unique extension property (though we do not need uniqueness here): 
    given the $C^\star$-envelope $(C^\star_{env}(\A_\lambda), \iota)$, there exists a representation $\Pi: C^\star_{env}(\A_\lambda) \to \B(\mathcal{K})$ such that $\Pi\circ\iota = \pi$, and the only completely positive map which agrees with $\Pi$ when restricted to $\iota(\A_\lambda)$ is $\Pi$ itself. 
    This yields that 
    $$\P_\cH (\Pi \circ \iota)(V_p) \vert_\cH 
    =
    \P_\cH \pi(V_p)\vert_\cH
    =
    \phi(V_p). $$
    Now consider the isomorphism $\Psi: \partial C^\star_\lambda (P) \to C^\star_{env}(\A_\lambda(P))$, that maps $q(V_p)$ to $\iota(V_p)$ by Theorem \ref{thm: boundary quotient is isomorphic to C*-envelope}. 
    Then, $\psi = \Pi \circ \Psi$ is a representation of $\partial C^\star_\lambda(P)$ on $\mathcal{K}$ that satisfies 
    $$(\psi \circ q)(V_p) = (\Pi \circ \Psi \circ q) (V_p) = (\Pi \circ \iota)(V_p) = \pi(V_p).$$

    For $(ii) \implies (i)$, note that 
    $\psi \circ \Psi^{-1}: C^\star_{env}(\A_\lambda) \to \B(\mathcal K)$ is a representation that maps $\iota (V_p)$ to $(\psi \circ q) (V_p) = \pi(V_p)$. 
    Moreover, since $\psi \circ \Psi^{-1}$ is a representation of a $C^\star$-algebra, it is completely contractive as well \cite{Pa02}. 
    Since $\iota$ is a complete isometry, 
    $\psi \circ \Psi^{-1} \circ \iota : \A_\lambda \to \B(\mathcal K)$ turns out to a complete contraction that agrees with $\pi$ on the span of $\{V_p \; \vert \; p \in P\}$, which is a dense subalgebra of $\A_\lambda$. 
    This implies that $\pi$ is a complete contraction, and $\phi$ satisfying that $\phi(V_p) = \P_\cH \pi (V_p) \vert_\cH$ must be a complete contraction as well. 
    This concludes the proof. 
\end{proof}

\begin{thm}\label{thm: representations with property (A) extend to the boundary}
    Let $P$ be a semigroup that is group-embeddable or right LCM cancellative, and let $T : P \to \B(\cH)$ be a representation. 
    Then, the following conditions are equivalent:
    \begin{enumerate}
        \item[(i)] $T$ has extension property (A), that is, $T$ extends to a completely contractive representation of $\A_\lambda(P)$. 
        \item[(ii)] $T$ has a dilation $\pi: P \to \B(\mathcal{K})$ for some $\mathcal{K} \supseteq \cH$ that induces a representation of $\partial C^\star_\lambda(P)$. 
        That is, we have a representation $\psi: \partial C^\star_\lambda (P) \to \B(\mathcal{K})$ with the property that if $q: C^\star_\lambda(P) \to \partial C^\star_\lambda(P)$ is the canonical quotient map, then 
        $$(\psi \circ q) (V_p) = \pi(p) \;\;\;\;\textrm{ for all } p \in P.$$
        Moreover, since $\pi$ is a dilation of $T$, 
        $$\P_\cH {\psi}(q(V_p) \vert_\cH 
        =
        \P_\cH \pi(p)\vert_\cH
        =
        T(p)  \;\;\;\;\textrm{ for all } p \in P. $$
    \end{enumerate}
\end{thm}

\begin{proof}
    Note that we may identify a representation $T: P \to \B(\cH)$ with a representation of the algebra generated by $\{V_p \: \vert \: p \in P\}$, which is a dense subalgebra of $\A_\lambda(P)$. 
    The rest follows from an application of Theorem \ref{thm: every completely contractive representation of the algebra also extends to the boundary}. 
\end{proof} 

These two theorems very directly lead us to the following dilation results:
\begin{thm}\label{thm: Ore semigroup representations with property (A) have unitary dilations}
    Let $P$ be a right LCM right Ore semigroup, and $T: P \to \B(\cH)$ be a representation. 
    Then, consider the following conditions:
    \begin{enumerate}
        \item[(i)] $T$ satisfies extension property (A).
        \item[(ii)] $T$ has a unitary dilation. 
    \end{enumerate}
    It is always true that $(i) \implies (ii)$. 
    If 
    $P$ is amenable, then $(ii) \implies (i)$ as well. 
\end{thm}

\begin{proof}
    First we prove $(i) \implies (ii)$. 
    Assume $T$ has extension property (A). 
    Then, by Theorem \ref{thm: representations with property (A) extend to the boundary}, we have a dilation $\pi: P \to \B(\K)$ where $\K \supseteq \cH$ and $T(p) = \P_\cH \pi(p)\vert_\K$ such that $\pi(p) = ( \psi \circ q ) (V_p)$ where $\psi:\partial C^\star_\lambda(P) \to \B(\K)$ where $\pi$ is a representation of $\partial C^\star_\lambda(P)$. 
    Now, since $P$ is Ore, recall from Theorem \ref{thm: boundary quotients for right-Ore semigroups} that $\partial C^\star_\lambda(P) \cong C^\star_\lambda(G)$ by the canonical isomorphism sending $q(V_p)$ to $\lambda_p$, where $G$ is the group generated by $P$ as $G= PP^{-1}$, and $\lambda_g$ is the left-regular representation of $g \in G$. 
    Since $\lambda_g$ is a unitary, so must be $q(V_p)$, and since $\psi$ is a representation of a $C^\star$-algebra, we conclude that $\pi(p)$ is a unitary as well. 

    For $(ii) \implies (i)$, consider a unitary dilation $\pi: P \to \B(\K)$ with $\K \supseteq \cH$ and $T(p) = \P_\cH \pi(p)\vert_\cH$ for all $p \in P$. 
    Since $P$ is Ore, $\pi$ induces a representation of the enveloping group $G$ as well, and consequently a representation of $C^\star(G)$. 
    When $G$ is amenable, $C^\star(G) \cong C^\star_\lambda(G)$. 
    Since $P$ is right Ore, $G = PP^{-1}$ is amenable if and only if $P$ is left-amenable. 
    Moreover, the Ore condition implies that $C^\star_\lambda(G) \cong \partial C^\star_\lambda(P)$ by Theorem \ref{thm: boundary quotients for right-Ore semigroups}. 
    Hence, by Theorem \ref{thm: representations with property (A) extend to the boundary}, we have that $T$ satisfies extension property (A). 
\end{proof}


One must mention that extension property (A) is not automatically guaranteed, that is, not every representation has this property. 
An example of a contractive representation of $\bN^3$ that fails to satisfy extension property (A) was provided in \cite{Pa70}, through a triple of commuting matrices which fail to satisfy a von Neumann inequality, and hence do not admit a unitary dilation. 
An interesting question now is to consider which semigroup representations have extension property (A). 

\begin{prop}\label{prop: representations where C*-algebras agree have property (A)}
    Let $P$ be a group-embeddable or right LCM cancellative semigroup, such that $C^\star_s(P) \cong C^\star_\lambda(P)$. 
    Then, every constructible representation of $P$ has extension property (A). 
\end{prop}

\begin{proof}
    By Definition \ref{defn: constructible representation}, every constructible representation of $P$ extends to a representation of $C^\star_s(P)$. 
    Since $C^\star_s(P) \cong C^\star_\lambda(P)$, this is also a representation of $C^\star_\lambda(P)$. 
    Since representations of $C^\star$-algebras are completely contractive, restricting to the operator algebra $\A_\lambda(P)$, we get a completely contractive representation, which is an extension of the representation of $P$ that we started with.  
\end{proof}

\begin{cor}\label{cor: amenable semigroups, constructible representations, and property (A)}
    Let $P$ be a group-embeddable or right LCM cancellative semigroup. 
    If $P$ is left-amenable, then every constructible representation has extension property (A). 
    If $P$ is right LCM and Nica-amenable, then every representation satisfying the condition in Inequality \ref{eq: Boyu's condition for isometric Nica-covariant dilation} has extension property (A). 
\end{cor}

\begin{proof}
    The proof in the left-amenable case follows from Theorem \ref{thm: lambda isomorphism from cancellative} and Proposition \ref{prop: representations where C*-algebras agree have property (A)}. 
    If $P$ is Nica-amenable, it follows from Definition \ref{defn: Nica-amenable semigroup} that every constructible representation extends to a representation of $C^\star_\lambda(P)$. 
    Now, the constructible representations are precisely the isometric Nica-covariant representations (Definition \ref{defn: Nica-covariant}), and it follows from Theorem \ref{thm: Boyu's theorem} that any representation satisfying the Inequality \ref{eq: Boyu's condition for isometric Nica-covariant dilation} admits an isometric Nica-covariant dilation. 
    The rest follows from Proposition \ref{prop: representations where C*-algebras agree have property (A)}. 
\end{proof}

\begin{rem}
    We would like to point out that the extension property (A) of a representation of a semigroup 
    holds in fact in much more general conditions than those outlined in Proposition \ref{prop: representations where C*-algebras agree have property (A)} and Corollary \ref{cor: amenable semigroups, constructible representations, and property (A)}. 
    However, complete contractivity of a representation of an operator algebra is quite difficult to verify in general. 
    A slightly easier problem to address might be to check contractivity instead of complete contractivity. 
    In Theorem 3.9, \cite{Pa02} shows that if the $\phi$ is a representation of an operator algebra $\A$, and $C^\star(\phi(\A))$ is commutative, then $\norm{\phi} = \norm{\phi}_{CB}$. 
    This in turn implies that if the representation of a semigroup is doubly commuting, then we only need to check contractivity of its extension to $\A_\lambda(P)$ for extension property (A). 
    This also includes the one-dimensional representations. 
\end{rem}

We now provide an alternate characterisation of extension property (A) using von Neumann inequalities, and demonstrate that our extension property (A) may be viewed as a generalisation of matricial von Neumann inequalities.

\begin{defn}\label{defn: von Neumann inequality}
    Let $T_1, T_2, \dots T_k$ be $k$ commuting contractions on a Hilbert space $\cH$. 
    We say that the $k$-tuple $(T_1, T_2, \dots T_k)$ satisfies a \textit{von Neumann inequality} if for every polynomial $p$ in $k$ variables, we have the inequality:
    $$\norm{p(T_1, T_2, \dots, T_k)} \leq \sup_{(z_1, z_2, \dots, z_k) \in {\T^k}}\abs{p(z_1, z_2, \dots , z_k)} = \norm{p}_{C(\T^k)}.$$
    Furthermore, we say that the $k$-tuple $(T_1, T_2, \dots T_k)$ satisfies a \textit{matricial von Neumann inequality} if for every $n$, for every set of polynomials $p_{i,j}$ in $k$ variables for $1 \leq i,j \leq n$, we have the inequality:
    $$\norm{\left(p_{i,j}(T_1, T_2, \dots, T_k)\right)_{i,j}} \leq \sup_{(z_1, z_2, \dots, z_k) \in {\T^k}}\norm{\left(p_{i,j}(z_1, z_2, \dots , z_k)\right)_{i,j}} = \norm{(p_{i,j})_{i,j}}_{M_n(C(\T^k))}.$$
\end{defn}


We state and sketch a proof of the following theorem, which is now straightforward, and parts of which are somewhat well known. 
\begin{thm}\label{thm: matricial von Neumann's inequality yields unitary dilations}
    Let $T_1, T_2, \dots T_k$ be $k$ commuting contractions on a Hilbert space $\cH$. 
    Then, the following are equivalent:
    \begin{enumerate}
        \item[(i)] There exist $k$ commuting unitaries $U_1, U_2, \dots U_k$ on a Hilbert space $\K \supseteq \cH$ satisfying 
        $$T_1^{n_1} T_2^{n_2} \dots T_k^{n_k} = \P_\cH U_1^{n_1} U_2^{n_2} \dots U_k^{n_k} \vert_\cH,$$
        for all $n_1, n_2, \dots n_k \in \bN$.
        \item[(ii)] The $k$-tuple $(T_1, T_2, \dots T_k)$ satisfies a matricial von Neumann inequality. 
        \item[(iii)] Let $\A \subset C(\T^k)$ be the (non-selfadjoint) operator algebra that is the linear span of polynomials on $\T^k$. 
        The map $\phi: \A \to \B(\cH)$, $z_1^{n_1} z_2^{n_2} \dots z_k^{n_k} \mapsto T_1^{n_1} T_2^{n_2} \dots T_k^{n_k}$ is a completely contractive representation.
        \item[(iv)] The representation $T: \bN^k \to \B(\cH)$, that maps the canonical generators $e_i$ to $T_i$ has the extension property (A). 
    \end{enumerate}
\end{thm}

\begin{proof}
    To prove $(i) \implies (ii)$, we just observe that for the commuting $k$-tuple of unitaries, they must be doubly commuting as well, and hence we may identify the $C^\star$-algebra they generate with $C(X)$ where $X$ is some compact subset of $\T^k$. 
    The rest is straightforward. 
    
    To prove $(ii) \implies (iii)$ is straightforward from the definition. 
    
    To prove $(iii) \implies (i)$, observe that since $\phi$ is unital completely contractive, so is its extension $\tilde{\phi}:\A + \A^\star \to \B(\cH)$ that maps $p+\bar{q} \mapsto \phi(p)+(\phi(q))^\star$ for all polynomials $p,q$. 
    Since every unital complete contraction on an operator system is unital completely positive \cite{Pa02}, we may now use Arveson's extension theorem to extend it to a unital completely positive map $\Phi: C(\T^k) \to \B(\cH)$. 
    The Stinespring dilation of this map will be a representation of $C(\T^k)$, and the unitary $U_i$ corresponding to the image of $z_i$ will be the unitary dilation of $T_i$. 

    To prove $(i) \iff (iv)$, note that since $\bN^k$ is Ore and amenable, $T$ having unitary dilation is equivalent to $T$ having extension property (A), by Theorem \ref{thm: Ore semigroup representations with property (A) have unitary dilations}. 
\end{proof}

Theorem \ref{thm: matricial von Neumann's inequality yields unitary dilations} brings things full circle in our paper: we demonstrate that the matricial von Neumann inequality is equivalent to extension property (A) for $\bN^k$, and unitary dilation is equivalent to dilation to the reduced boundary quotient for amenable Ore semigroups (of which $\bN^k$ are an example). 
Hence, we present extension property (A) as a generalisation of the matricial von Neumann inequality, and dilation to the reduced boundary quotient as a generalisation of unitary dilations, for semigroups more general than $\bN^k$. 
And just as the matricial von Neumann inequality is equivalent to existence of unitary dilations for $\bN^k$, our extension property (A) is equivalent to the existence of dilations to the boundary quotient. 
In this sense, our work is a true generalisation of Sz.-Nagy's dilation theorem and its counterparts. 



\begin{ex}
    To demonstrate the dilation phenomena we have proved in the general case, we provide a couple of well-known concrete examples 
    of semigroup representations which extend to completely contractive representations of the reduced semigroup operator algebra (i.e., satisfy extension property (A)) and hence the dilation property. 
    We refer the interested reader to \cite{Pa02} for an introduction to matricial von Neumann inequalities and unitary dilations, which are often an operator-theoretic tool in some of the dilation results stated below. 
    \begin{enumerate}
        \item[(i)] Any contractive representation of $\mathbb{N}$ has extension property (A). 
        This is equivalent to a matricial von Neumann's inequality which follows from Sz.-Nagy's dilation theorem \cite{SN53}, that says one may dilate every contraction on a Hilbert space to a unitary upto powers. 
        This unitary dilation induces a representation of $C(\T)$, which is the isomorphic to $\partial C^\star_\lambda(\mathbb{N})$. 

        \item[(ii)] Any contractive representation of $\mathbb{N}^2$ has extension property (A). 
        Let $T:\mathbb{N}^2 \to \B(\cH)$ be a contractive representation, that is, with the canonical generators of $\mathbb{N}^2$ as $e_1, e_2$, we have that $T_i = T(e_i)$ for $i=1,2$ are contractions such that $T_1 T_2 = T_2 T_1$. 
        By Ando's dilation theorem \cite{An63}, there exists a commuting pair of unitaries $U_1, U_2$ in $\B(\mathcal{K})$ with $\mathcal{K} \supseteq \cH$ such that for all $n_1,n_2 \in \mathbb{N}$, $T_1^{n_1} T_2^{n_2} = \P_\cH U_1^{n_1} U_2^{n_2} \vert_\cH$. 
        These unitaries help one obtain a matricial von Neumann inequality as well as induce a representation of $C(\T^2) \cong \partial C^\star_\lambda(\mathbb{N}^2)$. 

        \item[(iii)] Let $k \in \mathbb{N}$, and $T: \mathbb{N}^k \to M_2(\C)$ be a contractive representation, where $M_2(\C)$ denotes the space of all $2 \times 2$ matrices. 
        Then, $T$ has property (A). 
        A well known result by J. Holbrook \cite{Ho92} states that for every $k \in \mathbb{N}$, any $k$-tuple of commuting $2 \times 2$ matrices dilates to a $k$-tuple of commuting unitaries. 
        The rest follows, as this induces a unitary representation of $C(\T^k) \cong \partial C^\star_\lambda(\mathbb{N}^k)$, and also demonstrates extension property (A) through a matricial von Neumann's inequality. 
        \item[(iv)] Let $k \in \mathbb{N}$, and $V: \mathbb{N}^k \to \B(\cH)$ be an isometric representation. 
        Then, $V$ dilates to a unitary representation of $C(\T^k) \cong \partial C^\star_\lambda(\mathbb{N}^k)$, and hence also satisfies a matricial von Neumann's inequality, and thus extension property (A).
        We refer the reader to \cite{Pa02} for a proof that any $k$-tuple of commuting isometries dilates to a $k$-tuple of commuting unitaries.
    \end{enumerate}
\end{ex}

Note that none of the examples stated above, except for the $k=1$ case of (iv) actually satisfy either Nica-covariance or even Inequality \ref{eq: Boyu's condition for isometric Nica-covariant dilation} needed for a Nica-covariant dilation. 
This demonstrates our point that extension property (A) is in fact more general than the conditions required in \cite{LL22}. 
However, it is not trivially satisfied by just any contractive representation, as is demonstrated by the failure to obtain an Ando-type dilation result with $3$ or more contractions in general: S. Parrott \cite{Pa70} first gave an example of $3$ contractions which fail to satisfy a von Neumann inequality, and hence do not have extension property (A). 
How one might formulate a necessary condition for extension property (A) in terms of the generators of the representation generally, is a very interesting question for which we do not yet have a concrete answer. 

It is interesting to consider what the relations between the images of the generators of the dilation would be in the most general case. 
Note that one always has a canonical $\star$-homomorphism from the full crossed product to the reduced one, $\mathcal{Q}(P) \cong C(\partial \Omega_P) \rtimes G \to C(\partial \Omega_P) \rtimes_r G \cong \partial C^\star_\lambda (P)$ that maps the image of $v_p$ in the first to $q(V_p)$ in the second. 
Thus, in the right LCM case, by Equation \ref{eqn: foundation set}, $q(V_p) \in \partial C^\star_\lambda (P)$ are always going to satisfy the condition $\Pi_{f \in F}\left(1 - q(V_f) q(V_f)^\star\right) = 0$ for every foundation set $F$. 
This leads us to the following result:

\begin{cor}\label{cor: property A guarantees that there exists a dilation with boundary quotient relations}
    Let $P$ be a right LCM, group-embeddable semigroup, and let $T : P \to \B(\cH)$ be a representation with extension property (A) (i.e., it extends to a completely contractive representation of $\A_\lambda(P)$). 
    Then, $T$ has a dilation $\pi: P \to \B(\mathcal{K})$ for some $\mathcal{K} \supseteq \cH$ that satisfies 
    $$\Pi_{f \in F}\left(1_\mathcal{K} - \pi(f) \pi(f)^\star\right) = 0,$$
    for every foundation set $F \subset P$. 
\end{cor}

\begin{proof}
    The 
    proof follows from Corollary \ref{thm: representations with property (A) extend to the boundary}, and the discussion above. 
\end{proof}



\bibliographystyle{alpha}
\bibliography{bibliography.bib}

\end{document}